\documentclass[11pt]{article}
\usepackage{amsmath,amsthm,amsfonts,amssymb,amscd}
\usepackage{bbm}
\usepackage{url}
\usepackage{cite}
\usepackage{fullpage}
\usepackage{fancyhdr}
\usepackage{amsmath,amsthm,amsfonts,amssymb,amscd}
\usepackage{color}
\usepackage{float}
\usepackage{soul}
\usepackage{graphicx}
\usepackage[margin=1in]{geometry}

\theoremstyle{plain}
\newtheorem{theorem}{Theorem}[section]
\newtheorem{lemma}[theorem]{Lemma}
\newtheorem{corollary}[theorem]{Corollary}

\newtheorem{definition}{Definition}[section]
\newtheorem{remark}{Remark}[section]

\newtheorem{fact}[theorem]{Fact}
\usepackage{cite}
\theoremstyle{definition}

\newcommand{\la}{\langle}
\newcommand{\ra}{\rangle}

\newcommand{\nexto}{\kern -0.54em}

\newcommand{\dZ}{{\cal Z \kern -0.7em Z}}
\newcommand{\dC}{{\rm\hbox{C \kern-0.8em\raise0.2ex\hbox{\vrule height5.4pt width0.7pt}}}}
\newcommand{\dQ}{{\rm\hbox{Q \kern-0.85em\raise0.25ex\hbox{\vrule height5.4pt width0.7pt}}}}

\usepackage{epsfig}
\usepackage{latexsym}

\newcommand{\lqqd}{\hfill{$\Box$}\bigskip}

\newcommand{\NN}{\mathbb{N}}
\newcommand{\HH}{\mathcal{H}}
\newcommand{\RR}{\mathbbm{R}}

\newcommand{\dsty}{\displaystyle}

\begin{document}
\title{On proximal subgradient splitting method for minimizing the sum of two nonsmooth convex functions\footnote{This work was
partially supported by CNPq grants 303492/2013-9, 474160/2013-0 and
202677/2013-3 and by projects CAPES-MES-CUBA 226/2012 and UNIVERSAL FAPEG/CNPq.} }
\author{Jos\'e Yunier Bello Cruz\footnote{Institute of Mathematics and Statistics,
		Federal University of Goi\'as, Goi\^ania, Avenida Esperança, s/n Campus Samambaia,  Goi\^ania, GO, 74690-900, Brazil.
		E-mail: yunier@impa.br \&
yunier@ufg.br.} } \maketitle \vspace{6.0mm}
%\noindent\rule{\textwidth}{0.15mm} \vspace{2.0mm}
\begin{abstract}
\noindent In this paper we present a variant of the proximal
forward-backward splitting iteration for solving nonsmooth optimization
problems in Hilbert spaces, when the objective function is the sum
of two nondifferentiable convex functions. The proposed iteration,
which will be called Proximal Subgradient Splitting Method,
extends the classical subgradient iteration for important
classes of problems, exploiting the additive structure of the
objective function. The weak convergence of the generated sequence
was established using different stepsizes and under suitable
assumptions. Moreover, we analyze the complexity of the iterates.

\bigskip

\noindent{\bf Keywords:} Convex problems; Nonsmooth optimization
problems; Proximal forward-backward splitting iteration; Subgradient
method.

\medskip

\noindent{\bf Mathematical Subject Classification (2010):} 65K05,
90C25, 90C30.
\end{abstract}

%\noindent \rule{\textwidth}{0.15mm}
\section{ Introduction}
The purpose of this paper is to study the convergence properties of
a variant of the proximal forward-backward splitting method for
solving the following optimization problem:
\begin{equation}\label{prob}
\min \,f(x)+g(x) \;\mbox{ s.t. } \;x\in \HH,
\end{equation} where $\HH$ is a nontrivial real Hilbert space, and $f:\HH\rightarrow
\overline{\RR}:=\RR \cup
\{+\infty\}$ and $g:\HH\rightarrow\overline{\RR}$ are two proper lower
semicontinuous and convex functions. We are interested in the case
where both functions $f$ and $g$ are nondifferentiable, and when the
domain of $f$ contains the domain of
$g$. The solution set of this problem will be denoted by $S_*$,
which is a closed and convex subset of the domain of $g$. Problem
\eqref{prob} has recently been received much attention from the optimization community due to
its broad applications to several different areas such as
control, signal processing, system identification, machine learning
and restoration of images; see, for instance,
\cite{neal-boyd,comb-2005,comb-2011, geman} and the references
therein.

A special case of problem \eqref{prob} is the nonsmooth constrained
optimization problem, taking $g=\delta_C$ where $\delta_C$ is the
indicator function of a nonempty closed and convex set $C$ in $\HH$,
defined by $\delta_C(y):=0$, if $y\in C$ and $+\infty$, otherwise.
Then, problem \eqref{prob} reduces to the constrained minimization
problem
\begin{equation}\label{rest-opt}
\min\; f(x) \;\mbox{ s.t. } \;x\in C.
\end{equation}
Another important case of problem \eqref{prob}, which has had much
interest in signal denoising and data mining, is the following
optimization problem with $\ell_1$-regularization
\begin{equation}\label{l1-reg}
\min\; f(x)+\lambda \|x\|_1\;\mbox{ s.t. } \;x\in \HH,
\end{equation} where $\lambda>0$ and the norm $\|\cdot\|_1$ is used to induce the
sparsity in the solutions. Moreover, problem \eqref{l1-reg} covers
the important and well studied Low-Rank problem, when $\HH=\RR^n$
and $f(x)=\|Ax-b\|_2^2$ where $A\in \RR^{m\times n}$, $m<<n$, and
$b\in \RR^m$, which is just a convex approximation of the very
famous $\ell_0$ minimization problem; see \cite{candes-tao}.
Recently, this problem became popular in signal processing and
statistical inference; see, for instance, \cite{Tropp,Fig}.

We focus here on the so-called {\em proximal
forward-backward splitting iteration} \cite{neal-boyd}, which
contains a forward gradient step of $f$ (an explicit step) followed
by a backward proximal step of $g$ (an implicit step). The main idea
of our approach consists of replacing, in the forward step of the
proximal forward-backward splitting iteration, the gradient of $f$
by a subgradient of $f$ (note that here $f$ is assumed
nondifferentiable in general). In the particular case that $g$ is
the indicator function, the proposed iteration reduces to the
classical projected subgradient iteration.

To describe and motivate our iteration, first we recall the definition of the
so-called {\em proximal operator} as ${\bf prox}_g: \HH\rightarrow \HH$ associated to $g$ a proper lower semicontinuous convex function, where
${\bf prox}_g(z)$, $z\in \HH$ is the unique solution of the following strongly convex
optimization problem
\begin{equation}\label{prox}
\min\; g(y)+\frac{1}{2}\|y-z\|^2  \;\mbox{ s.t. } \;y\in \HH.
\end{equation}
Note that the norm $\|\cdot\|$ is induced by this inner product of $\HH$, \emph{i.e.}, $\|x\|:=\sqrt{\la x,x\ra}$ for all $x\in\HH$. The proximal operator ${\bf prox}_g$ is well-defined and has many
attractive properties, \emph{e.g.},  it is continuous and firmly
nonexpansive, \emph{i.e.}, for all $x,y\in\HH$, $\|{\bf prox}_g(x)-{\bf prox}_g(y)\|^2\le \|x-y\|^2-\|[x-{\bf prox}_g(x)]-[y-{\bf prox}_g(y)]\|^2$. This nice property can be used to construct algorithms to solve optimization problems \cite{rock-1976}; for other properties and algebraic rules see
\cite{comb-2005,comb-2011, librobauch}. If $g=\delta_C$ is the
indicator function, the orthogonal projection onto $C$,
${\bf P}_C(x):=\left\{y\in C : \|x-y\|={\rm dist}(x,C)\right\}$ is the same
as ${\bf prox}_{\delta_C}(x)$ for all $x\in \HH$ \cite{bb}; For an exhaustive discussion about the evaluation of the proximity operator of a wide variety of functions see Section $6$ of \cite{neal-boyd}. Now, let us recall the definition of the
subdifferential operator $\partial g:\HH\rightrightarrows\HH$ by $
\partial g(x):=\left\{w\in \HH : g(y)\geq g(x)+\la w,y-x\ra,\; \forall\, y\in
\HH\right\}.$ We also present the relation of the proximal operator ${\bf
prox}_{\alpha g}$ with the subdifferential
operator $\partial g$, \emph{i.e.}, ${\bf prox}_{\alpha g}=({\rm Id}+\alpha
\partial g)^{-1}$ and as a direct consequence of the
first optimality condition of \eqref{prox}, we have the following useful
inclusion:
\begin{equation}\label{in-sub}
\frac{z-{\bf prox}_{\alpha g}(z)}{\alpha}\in \partial g({\bf
prox}_{\alpha g}(z)),
\end{equation} for any $z\in \HH$ and
$\alpha>0$. The iteration proposed in this paper, called {\em Proximal
Subgradient Splitting Method}, is
motivated by the well-known fact that $x\in S_*$ if and only if there
exists  $u\in
\partial f(x)$ such that $x={\bf prox}_{\alpha g}(x-\alpha u)$. Thus, the iteration generalizes the proximal forward-backward
splitting iteration for the differentiable case, as a fixed point
iteration of the above equation, which is defined as follows: stating at $x^0$
belonging to the domain of $g$, set
\begin{equation}\label{F-B}x^{k+1}={\bf prox}_{\alpha_k g}(x^k-\alpha_k u^k), \end{equation}
where $u^k\in \partial f(x^k)$ and the stepsize $\alpha_k$ is
positive for all $k\in\NN$. Iteration \eqref{F-B} recovers the classical subgradient iteration \cite{poljak}, when $g=0$,  and the proximal point iteration \cite{rock-1976}, when $f=0$. Moreover, it covers important situations in which $f$ is
nondifferentiable and it can  also be seen as a {\em forward-backward Euler discretization} of the subgradient flow differential inclusion 
$$
\dot{x}(t)\in-\partial [f(x(t))+g(x(t))]
$$ with variable $x:\RR_+\to\HH$; see \cite{neal-boyd}. Actually, if the derivative on the left side is replaced by the divided difference $(x^{k+1}-x^k)/\alpha_k$, then the discretization obtained is
$
(x^k-x^{k+1})/\alpha_k\in \partial f(x^k)+\partial g(x^{k+1}),
$ which is the proximal subgradient iteration \eqref{F-B}.

The nondifferentiability of the function $f$ has
a direct impact on the computational effort and the importance of
such problems, when $f$ is nonsmooth, is underlined because they occur frequently in
applications. Nondifferentiability arises, for instance,
in the problem of minimizing the total variation of a signal over a
convex set, in the problem of minimizing the sum of two set-distance
functions, in problems involving maxima of convex functions, the
Dantzing selector-type problems, the non-Gaussian image denoising
problem and in Tykhonov regularization problems with $L_1$ norms;
see, for instance, \cite{D-R-comb,james,tikhonov}. The iteration of the proximal subgradient splitting method, proposed in \eqref{F-B},
can be applied in these important instances, extending the classical 
subgradient iteration for more general problems as \eqref{l1-reg}.
In problem \eqref{prob}, $f$ is usually assumed to be
differentiable as in \cite{nesterov-2013}, which is not necessarily the case in this work. Moreover, the convergence of the iteration \eqref{F-B} to a
solution of \eqref{prob} has been established in the literature,
when the gradient of $f$ is globally Lipschitz continuous and
the stepsizes $\alpha_k$, $k\in\NN$ have to be chosen very small, \emph{i.e.}, for all $k$, $\alpha_k$ is less than
some constant related with the Lipschitz constant of the gradient of $f$; see, for instance,
\cite{comb-2005}. Recently, when $f$ is continuously differentiable but the Lipschitz constant is not available, the steplengths can be chosen using backtracking procedures; see
\cite{beck-teu,neal-boyd,nesterov-2013,yun-nghia-2014}. 

It
is important to mention that the forward-backward iteration finds
also applications in solving more general problems, like the
variational inequality and inclusion problems; see, for instance,
\cite{yunier-iusem,benar,bot-2014,comb-2004, rock-97} and the
references therein. On the other hand, the standard convergence
analysis of this iteration, for solving these general problems, requires at least a co-coercivity
assumption of the operator and the stepsizes to lie within a suitable interval; see, for
instance, Theorem $25.8$ of \cite{librobauch}. Note that co-coercive
operators are monotone and Lipschitz continuous, but the converse
does not hold in general; see \cite{Zhu-Marcotte-SIAM-1996}.
Although, for gradients of lower semicontinuous, proper and convex
functions, the co-coercivity is equivalent to the global Lipschitz
continuity assumption. This nice and surprising fact, which is strongly used 
in the convergence analysis of the proximal forward-backward method
for problem \eqref{prob}, when $f$ is differentiable, is known as
the {\em Baillon-Haddad Theorem}; see Corollary $18.16$ of
\cite{librobauch}.

The main aim of this work is release the differentiability of $f$
of the forward-backward splitting method, extending the classical
projected subgradient method and 
containing, as particular case, a new proximal subgradient iteration
for more general problems.  

This work is organized as follows. The next subsection provides our
notations and assumptions, and some preliminaries results that will
be used in the remainder of this paper. The proximal subgradient
splitting method and its weak convergence are analyzed by choosing
different stepsizes in Section \ref{Sec-new-4}. Finally, Section
\ref{s:6} gives some concluding remarks.

\subsection{Assumptions and Preliminaries}
In this section, we present our assumptions, classical
definitions and some results needed for the convergence analysis of the
proposed method.

We start by recalling some definitions and notation used in this paper,
which are standard and follows from \cite{librobauch,neal-boyd}. Throughout
this paper, we write $p:=q$ to indicate that $p$ is defined to be
equal to $q$. We
write $\NN$ for the nonnegative integers $\{0, 1, 2,\ldots\}$ and remind that
the extended-real number system is $\overline{\RR} := \RR \cup
\{+\infty\}$. The closed ball centered at $x\in\HH$ with radius
$\gamma>0$ will be denoted by $\mathbb{B}[x;\gamma]$, \emph{i.e.},
$\mathbb{B}[x;\gamma]:=\{y\in\HH : \|y-x\|\leq\gamma\}$. The domain
of any function $h:\HH\rightarrow\overline{\RR}$, denoted by ${\bf dom}(h)$, is
defined as ${\bf dom}(h):=\{x\in\HH : h(x)<+\infty\}$. The optimal value
of problem \eqref{prob} will be denoted by $s_*:=\dsty\inf\{(f+g)(x) :  x\in\HH\}$, noting that when $S_*\neq \emptyset$, $s_*=\dsty\min\{(f+g)(x) :  x\in\HH\}=(f+g)(x_*)$ for any $x_*\in S_*$. Finally, $\ell_1(\NN)$ denotes the set of summable
sequences in $[0, +\infty)$.

\noindent Throughout this paper we assume the following:

\noindent \; {\bf A1.} $\partial f$ is bounded on bounded sets on the domain of $g$, \emph{i.e.},
$\exists\,$$\zeta>0$ such that $\partial f(x)\subseteq \mathbb{B}[0;\zeta]$ for all $x\in V$, where $V$ is any bounded and closed subset of ${\bf dom}(g)$.

\noindent \; {\bf A2.} $\partial g$ has bounded elements on the domain of $g$, \emph{i.e.}, $\exists\,$$\rho\ge0$ such that
$\partial g(x)\cap\mathbb{B}[0;\rho]
\neq\emptyset$ for all $x\in {\bf dom}(g)$.

In connection with Assumption {\bf A1}, we recall that $\partial f$
is locally bounded on its open domain. In finite dimension spaces,
this result implies that {\bf A1} always holds when ${\bf dom}(f)$ is open. A widely used sufficient condition for  {\bf A1}  is the Lipschitz continuity of $f$ on ${\bf dom}(g)$. Furthermore, the
boundedness of the subgradients is crucial for the convergence
analysis of many classical subgradient methods in Hilbert spaces and
it has been widely considered in the literature; see, for instance,
\cite{AIS,sc-1,yunier-iusem, poljak}. 

Regarding Assumption {\bf
A2}, we emphasize that it holds trivially for important instances of
problem \eqref{prob}, \emph{e.g.}, problems \eqref{rest-opt} and
\eqref{l1-reg} because $\partial \delta_C(x)=N_C(x)$ and $\partial \|x\|_1=\{u\in\HH : \|u\|_\infty\le1\,,\; \la u,x\ra=\|x\|_1 \, \}$, respectively, or when ${\bf dom}(g)$ is a bounded set or also when $\HH$ is a
finite dimensional space. Note that Assumption {\bf A2} allows instances where $\partial g$ is an
unbounded set as is the particular case when $g$ is the
indicator function.
It is an existence condition, which is in general weaker than {\bf A1}.

Let us end the section by recalling the well-known concepts
so-called quasi-Fej\'er and Fej\'er convergence. 
\begin{definition}\label{def-fejer}
Let $S$ be a nonempty subset of $\HH$. A sequence $(x^k)_{k\in \NN}$
	in $\HH$ is said to be quasi-Fej\'er convergent to $S$ if and only
	if for all $x \in S$ there exists a sequence $(\epsilon_k)_{k\in
		\NN}$ in $\ell_1(\NN)$ and
	$\| x^{k+1}-x\|^2 \leq \| x^{k}-x\|^2 +\epsilon_k$ for all
	$k\in\NN$. When $(\epsilon_k)_{k\in \NN}$ is a null sequence, we say
	that $(x^k)_{k\in \NN}$ is Fej\'er convergent to $S$.
\end{definition}
The definition
originates in \cite{Ermolev} and has been elaborated further in
\cite{comb-2001}.
This definition, originated in
\cite{Ermolev}, has been elaborated further in \cite{comb-2001}.
In the following we present two well-known fact for quasi-Fej\'er
convergent sequences.
\begin{fact}\label{fato-qF} If the sequence $(x^k)_{k\in \NN}$ is quasi-Fej\'er convergent
to $S$, then:
\item[ {\bf(a)}] The sequence $(x^k)_{k\in \NN}$ is bounded.
\item[ {\bf(b)}] $(x^k)_{k\in \NN}$ is weakly convergent iff
all weak accumulation points of $(x^k)_{k\in \NN}$ belong
to $S$.
\end{fact}
\begin{proof}
Item {\bf(a)} follows from Proposition $3.3$(i) of \cite{comb-2001}, and
Item {\bf(b)} follows from Theorem $3.8$ of \cite{comb-2001}.
\end{proof}

%%%%%%%%%%%%%%%%%%%%%%%%%%%%%%%%%%%%%%%%%%%%%%%%%%%%%%%%%%%%%%%%%%%%%%%%%%%%%%%%%%%%
%%%%%%%%%%%%%%%%%%%%%%%%%%%%%%%% Algorithm 1 %%%%%%%%%%%%%%%%%%%%%%%%%%%%%%%%%%%%%%%
%%%%%%%%%%%%%%%%%%%%%%%%%%%%%%%%%%%%%%%%%%%%%%%%%%%%%%%%%%%%%%%%%%%%%%%%%%%%%%%%%%%%

\section{The Proximal Subgradient Splitting Method}\label{Sec-new-4}

In this section we propose the proximal subgradient splitting method
extending the classical subgradient iteration. We prove that the
sequence of points generated by the proposed method converges weakly
to a solution of \eqref{prob} using different strategies for choosing of the
stepsizes. Moreover, we show the complexity analysis for the generated sequence.

\noindent The method is formally stated as follows:
\begin{center}\fbox{\begin{minipage}[b]{\textwidth}
%\begin{Mmethod}[{\bf P}roximal {\bf S}ubgradient {\bf S}plitting {\bf M}ethod]\label{A2}
%\begin{retraitsimple}
{\bf P}roximal {\bf S}ubgradient {\bf S}plitting Method ({\bf PSS
Method})
\item [    ] {\bf Initialization Step.} Take $x^0\in {\bf dom}(g)$.

\item [    ] \noindent {\bf Iterative Step.} Set
\begin{equation*}\label{iteracion-sub}\tag{6}x^{k+1}={\bf prox}_{\alpha_k g}\left(x^k-\alpha_k u^k\right),\end{equation*}
where $u^k\in \partial f(x^k)$.

\item [    ] \noindent  {\bf Stop Criteria.} If $x^{k+1}=x^k$ then stop.
%\end{Mmethod}
\end{minipage}}\end{center}
If {\bf PSS Method} stops at step $k$, then $x^k={\bf
prox}_{\alpha_k g}\left(x^k-\alpha_k u^k\right)$ with
$u^k\in\partial f(x^k)$, implying that $x^k$ is solution of problem
\eqref{prob}. Then, from now on, we assume that {\bf PSS
 Method} generates an infinite sequence $(x^k)_{k\in
\NN}$. Moreover, it follows directly from \eqref{F-B} that the
sequence $(x^k)_{k\in \NN}$ belongs to ${\bf dom}(g)$.

Before the formal analysis of the convergence properties of {\bf PSS Method}, we discuss below about the necessity of taking a (forward) subgradient step of $f$ instead of another (backward) proximal step.
\begin{remark} {\rm 
	To evaluate 
	the proximal operator of $f$ is necessary to solve the strongly convex
	minimization problem as \eqref{prox}. Thus, in the context of problem \eqref{prob},
	we assume that it is hard to evaluate the proximal operator of $f$, leaving out the possibility to use the standard and very powerful iteration so-called {\em Douglas-Rachford
		splitting method} presented in \cite{D-R-comb}. Such situations appear mainly when $f$ has a complicated algebraic expression and therefore it may impossibility to solve, explicitly or efficiently, subproblem \eqref{prox}. Indeed, very often in the applications,
	the formula for the proximity operator is not available in closed form and ad
	hoc algorithms have be used to compute ${\bf prox}_{\alpha f}$.
	This happens for instance
	when applying proximal methods to image deblurring with total variation \cite{beck-te},
	or to structured sparsity regularization problems in machine learning and inverse problems \cite{villa}. 
	
	A classical problem of the form of \eqref{prob}, when the subgradient of $f$ is  easily  available and ${\bf prox}_f$ does not has explicitly formula is the dual formulation of the following constrained convex problem:
	\begin{equation}\label{probRemark}
	\min h_0(y) \quad \mbox{subject that} \quad h_i(y)\le 0 \;\; (i=1, \ldots, n), 
	\end{equation} where $h_i:\RR^m\to \RR$ $(i=0, \ldots, n)$ are convex. 
	It can be writen as 
	$$
	\min_{x\in\RR^n} f(x) + \delta_{\RR^n_+}(x)
	$$
	with $f: \RR^n\to\RR$ defined as $f(x)=-\inf_{y\in\RR^n}\{h_0(y)+\sum_{i=1}^n x_ih_i(y)\}$. It is well-known that $$\partial f(x)={\rm conv}\left\{ h(y_x) : f(x)=h_0(y_x)+\sum_{i=1}^n x_ih_i(y_x)\right\},$$ and
	${\rm conv}\{S\}$ denotes the convex hull of a set $S$. However, compute ${\bf prox}_f$ does not look an easy problem. This argument is used widely in the literature to motivated the projected subgradient method, which can be easily modified for recovering  problems as \eqref{prob}, when $g$ is not necessary the indicator function. Indeed, consider problem \eqref{probRemark} when $n=m$ with an additional and simple restriction $g_0$, that is:
	\begin{equation}\label{probRemark*}
	\min h_0(y) \quad \mbox{subject that} \quad h_i(y)\le 0 \;\; (i=1, \ldots, n), \; g_0(y)\le 0, 
	\end{equation} which can be dualized as 
		$$
		\min_{x\in\RR^n} f(x) + \delta_{\RR^n_+}(x) + \lambda g_0(x), \qquad \lambda>0.
		$$ This problem is a particular case of \eqref{prob}, by taking $g=\delta_{\RR^n_+} + \lambda g_0$. Note that if ${\bf dom}( g_0)\subseteq \RR^n_+$ then $g=\lambda g_0$.
		
Thus, {\bf PSS Method} uses the proximal operator of $g$ and the explicit
	subgradient iteration of $f$ (\emph{i.e.}, the proximal operator of $f$ is never
	evaluated), which is, in general, much easier to implement than the proximal
	operator of $f+g$ or $f$, as happens in the standard proximal point iteration
	or the Douglas-Rachford algorithm, respectively for solving nonsmooth problems,
	as \eqref{prob}; see, for instance, \cite{D-R-comb}. Furthermore, note that in our case the subgradient iteration for the sum $f+g$ is not possible, because the domains of $f$ and $g$ are not the whole space.}\lqqd
\end{remark}
\noindent In the following we prove a crucial property of the
iterates generated by {\bf PSS
 Method}.
\begin{lemma} \label{lema-para-qF1}
Let $(x^k)_{k\in\NN}$ and $(u^k)_{k\in\NN}$ be the sequences generated by {\bf PSS
Method}. Then, for all $k\in \NN$ and $x\in {\bf dom}(g)$,
$$\|x^{k+1}-x\|^2\le\|x^{k}-x\|^2
+2\alpha_k\left[(f+g)(x)-(f+g)(x^k)\right]+\alpha_k^2\|u^k+w^k\|^2,$$
where $w^k\in \partial g(x^k)$ is arbitrary.
\end{lemma}
\begin{proof}
Take any $x\in {\bf dom}(g)$. Note that \eqref{in-sub} and \eqref{F-B}
imply that $\dsty \bar{w}^{k+1}:=\frac{x^k-x^{k+1}}{\alpha_k}-u^k$, with
$u^k\in
\partial f(x^k)$ as defined by {\bf PSS
 Method}, belongs to $\partial
g(x^{k+1})$. Then,
\begin{align*}
\,&\alpha_k^2\|u^k+\bar{w}^{k+1}\|^2+\|x^{k}-x\|^2 -\|x^{k+1}-x\|^2
=\|x^{k+1}-x^k\|^2+\|x^{k}-x\|^2 -\|x^{k+1}-x\|^2\\\,&=2\la
x^{k}-x^{k+1},x^k-x\ra
=\, 2\alpha_k\la u^k,x^k-x\ra +2\la x^k-x^{k+1}-\alpha_ku^k,x^k-x\ra \\
\,&=\,2\alpha_k\la u^k,x^k-x\ra +2\alpha_k\left \la
\frac{x^k-x^{k+1}}{\alpha_k}-u^k,x^{k+1}-x\right\ra+2\alpha_k\left\la
\frac{x^k-x^{k+1}}{\alpha_k}-u^k,
x^k-x^{k+1}\right\ra\\
\,&=\,2\alpha_k\la u^k,x^k-x\ra +2\alpha_k\left \la
\bar{w}^{k+1},x^{k+1}-x\right\ra+2\alpha_k\la u^k, x^{k+1}-x^k\ra+
2\|x^k-x^{k+1}\|^2.
\end{align*}
Now using again that $\dsty
\frac{x^k-x^{k+1}}{\alpha_k}-u^k=\bar{w}^{k+1}\in
\partial g(x^{k+1})$ and the convexity of $g$ and $f$, we obtain
\begin{align*} \,&2\la x^{k}-x^{k+1},x^k-x\ra\ge 2\alpha_k\left[f(x^k)-f(x)+ g(x^{k+1})-g(x)+\la u^k, x^{k+1}-x^k\ra\right]
+2
\|x^k-x^{k+1}\|^2\\
\,&= 2\alpha_k\left[(f+g)(x^k)-(f+g)(x)+g(x^{k+1})-g(x^{k})+\la u^k,
x^{k+1}-x^k\ra\right]+2\alpha_k^2\|u^k+\bar{w}^{k+1}\|^2\\\,&\ge
2\alpha_k\left[(f+g)(x^k)-(f+g)(x)+\la w^k+u^k,
x^{k+1}-x^k\ra\right]+2\alpha_k^2\|u^k+\bar{w}^{k+1}\|^2,
\end{align*} for any $w^k\in \partial g(x^k)$.
We thus have shown that
\begin{align*}\|x^{k+1}-x\|^2\le \,&\|x^{k}-x\|^2
+2\alpha_k\left[(f+g)(x)-(f+g)(x^k)\right]\\\,&+ 2\alpha_k^2 \la
u^k+w^k,
u^k+\bar{w}^{k+1}\ra-\alpha_k^2\|u^k+\bar{w}^{k+1}\|^2\\=\,&\|x^{k}-x\|^2
+2\alpha_k\left[(f+g)(x)-(f+g)(x^k)\right]+
\alpha_k^2\|u^k+w^{k}\|^2-\alpha_k^2\|w^k-\bar{w}^{k+1}\|^2.
\end{align*}
Note that $w^k\in \partial g(x^k)$ is arbitrary and the result follows.
\end{proof}
Since subgradient methods are not descent methods, as the
proposed method here, it is common to keep track of the best point
found so far, \emph{i.e.}, the one with minimum function value among the iterates. At each
step, we set it recursively as $(f+g)^0_{\rm best}:=(f+g)(x^0)$ and
\begin{equation}\label{mejor-valor}(f+g)^k_{\rm best}:=
\min\left\{(f+g)^{k-1}_{\rm best}, (f+g)(x^k)\right\},\end{equation}
for all $k$. Since $\left((f+g)^k_{\rm best}\right)_{k\in \NN}$  is
a decreasing sequence, it has a limit (which can be $-\infty$). When
the function $f$ is differentiable and its gradient Lipschitz
continuous, it is possible to prove the complexity of the iterates
generated by {\bf PSS Method}; see \cite{nesterov-2013}. In our
instance ($f$ is not necessarily differentiable) we expect, of
course, slower convergence.

Next we present a convergence rate result for the sequence of the best
functional values $\left((f+g)^k_{\rm best}\right)_{k\in \NN}$ to
$\min\{(f+g)(x) :  x\in\HH\}$.
\begin{lemma}\label{f-best-alphak} Let $\left((f+g)^k_{\rm best}\right)_{k\in \NN}$ be the sequence defined by \eqref{mejor-valor}.
If $S_*\neq\emptyset$ then, for all $k\in \NN$,
\begin{equation*}
(f+g)^k_{\rm best}-\min_{x\in\HH}(f+g)(x)\le \frac{[{\rm
dist}(x^0,S_*)]^2+C_k\sum_{i=0}^k \alpha_i^2}{ 2\sum_{i=0}^k
\alpha_i},
\end{equation*} where $C_k:=\max\left\{\|u^i+w^i\|^2 : 0\le i\le
k \right\}$ with $w^i\in \partial g(x^i)$
$(i=0,\ldots,k)$ are arbitrary.
\end{lemma}
\begin{proof}
Define $x_*:={\bf P}_{S_*}(x^0)$. Note that $x_*$ exists because $S_*$ is
a nonempty closed and convex set of $\HH$. By applying  Lemma
\ref{lema-para-qF1}, $k+1$ times, for $i\in\{0,1,\ldots,k\}$ at $x_*\in S_*$, we get
\begin{align}\label{para-lema-ergodic}\nonumber
\|x^{k+1}-x_*\|^2\le\,&\|x^{k}-x_*\|^2
+2\alpha_k\left[(f+g)(x_*)-(f+g)(x^k)\right]+\alpha_k^2\|u^k+w^k\|^2\\\le\,&\|x^{0}-x_*\|^2
+2\sum_{i=0}^k
\alpha_i\left[(f+g)(x_*)-(f+g)(x^i)\right]+\sum_{i=0}^k\alpha_i^2\|u^i+w^i\|^2\\\le\,&[{\rm
dist}(x^0,S_*)]^2 +2\left[\min_{x\in\HH}(f+g)(x)-(f+g)^k_{\rm best}\right]\sum_{i=0}^k
\alpha_i+C_k\sum_{i=0}^k\alpha_i^2,\nonumber
\end{align} where $\left(f+g\right)^k_{\rm best}$ is defined by \eqref{mejor-valor} and the result follows after simple algebra.
\end{proof}
Next we establish the rate of convergence of the ergodic
sequence $(\bar{x}^k)_{k\in \NN}$ of
$(x^k)_{k\in \NN}$, which is defined recursively as
$\bar{x}^0=x^0$ and given $\sigma_0=\alpha_0$ and
$\sigma_k=\sigma_{k-1}+\alpha_k$, we define
$$
\bar{x}^k=\left(1-\frac{\alpha_k}{\sigma_k}\right)\bar{x}^{k-1}+\frac{\alpha_k}{\sigma_k}x^{k}.$$
After easy induction, we have $\sigma_k=\sum_{i=0}^k\alpha_i$ and
\begin{equation}\label{egodic-seq} \bar{x}^k=\frac{1}{\sigma_k}\sum_{i=0}^k
\alpha_i\, x^i,
\end{equation} for all $k\in \NN$.

\noindent The following result is very similar to Lemma
\ref{f-best-alphak}, considering the ergodic sequence defined by \eqref{egodic-seq}.
\begin{lemma}\label{f-best-alphak2}
Let $(\bar{x}^k)_{k\in \NN}$ be the ergodic sequence
defined by \eqref{egodic-seq}. If $S_*\neq \emptyset$, then
\begin{equation*}
(f+g)(\bar{x}^k)-\min_{x\in\HH}(f+g)(x)\le \frac{[{\rm
dist}(x^0,S_*)]^2+C_k\sum_{i=0}^k \alpha_i^2}{ 2\sum_{i=0}^k
\alpha_i},
\end{equation*} where $C_k=\max\left\{\|u^i+w^i\|^2 : 0\le i\le
k \right\}$ with $w^i\in \partial g(x^i)$
$(i=0,\ldots,k)$ are arbitrary.
\end{lemma}
\begin{proof}
Proceeding as in the proof of Lemma \ref{f-best-alphak} until Equation
\eqref{para-lema-ergodic} and after dividing by
$\sum_{i=0}^k \alpha_i$, we get
\begin{align}\label{para-lema-ergodic2}\nonumber
\sum_{i=0}^k
\frac{\alpha_i}{\sigma_k}\left[(f+g)(x^i)-\min_{x\in\HH}(f+g)(x)\right]\le\,&\frac{1}{2\sigma_k}\left([{\rm
dist}(x^0,S_*)]^2-\|x^{k+1}-x_*\|^2\right)
+\frac{C_k}{2\sigma_k}\sum_{i=0}^k\alpha_i^2\\\le\,&\frac{1}{2\sigma_k}\left([{\rm
dist}(x^0,S_*)]^2 +C_k\sum_{i=0}^k\alpha_i^2\right),
\end{align} where $\sigma_k:=\sum_{i=0}^k \alpha_i$. Using the convexity of $f+g$ after note that $\frac{\alpha_i}{\sigma_k}\in [0,1]$ for all $i\in\{0,1,\ldots,k\}$ and $\sum_{i=0}^k\frac{\alpha_i}{\sigma_k}=1$ and \eqref{egodic-seq} in the above inequality \eqref{para-lema-ergodic2},
the result follows.
\end{proof}
Next we focus on constant step sizes, which is motivated
by the fact that we are interested in quantifying the progress of
the proposed method to find an
approximate solution.
\begin{corollary}\label{pasos-const} Let $(x^k)_{k\in \NN}$ be the sequence generated by {\bf PSS
 Method} with the stepsizes $\alpha_k$ constant equal to $\alpha$, $\left((f+g)^k_{\rm best}\right)_{k\in \NN}$ be the sequence defined by \eqref{mejor-valor}
 and $(\bar{x}^k)_{k\in\NN}$
be the ergodic sequence as \eqref{egodic-seq}. Then, the iteration
attains the optimal rate at $\alpha=\frac{{\rm
dist}(x^0,S_*)}{\sqrt{C_k}}\cdot\frac{1}{\sqrt{k+1}}$, \emph{i.e.}, for all
$k\in \NN$,
\begin{equation*}
(f+g)^k_{\rm best}-\min_{x\in\HH}(f+g)(x)\le \frac{[{\rm dist}(x^0,S_*)]^2+\alpha^2 (k+1)
C_k}{ 2(k+1) \alpha}\le \frac{{\rm dist}(x^0,S_*)\cdot\sqrt{C_k}}{\sqrt{k+1}}
\end{equation*}
and
\begin{equation*}
(f+g)(\bar{x}^k)-\min_{x\in\HH}(f+g)(x)\le \frac{[{\rm dist}(x^0,S_*)]^2+\alpha^2 (k+1)
C_k}{ 2(k+1) \alpha}\le \frac{{\rm dist}(x^0,S_*)\cdot\sqrt{C_k}}{\sqrt{k+1}},
\end{equation*} where $C_k=\max\left\{\|u^i+w^i\|^2 : 0\le i\le
k \right\}$ with $w^i\in \partial g(x^i)$ $(i=0,\ldots,k)$ are
arbitrary.
\end{corollary}
\begin{proof}
If we consider constant stepsizes, \emph{i.e.}, $\alpha_k=\alpha$ for all
$k\in\NN$, then the optimal rate is obtained when
$\alpha=\frac{{\rm
dist}(x^0,S_*)}{\sqrt{C_k}}\cdot\frac{1}{\sqrt{k+1}}$ from
minimizing the right part of Lemmas \ref{f-best-alphak} and
\ref{f-best-alphak2}. 
\end{proof}
Note that under Assumption {\bf A2}, $C_k\le (\max_{1\le i\le k}\|u^i\|+\rho)^2$. Hence when ${\bf dom}(g)$ is bounded, Assumption {\bf A1} implies that $C_k\le(\zeta+\rho)^2$ for all $k\in\NN$. In this case our analysis showed that the expected error of the iterates
generated by {\bf PSS
 Method} with constant stepsizes after $k$
iterations is $\mathcal{O}\left((k+1)^{-1/2}\right)$. Hence, we can
search an $\varepsilon$-solution of problem \eqref{prob} with
$\mathcal{O}\left(\varepsilon^{-2}\right)$ iterations. Of course,
this is worse than the rate $\mathcal{O}(k^{-1})$ and
$\mathcal{O}\left(\varepsilon^{-1}\right)$ iterations of the
proximal forward-backward iteration for the differentiable and
convex $f$ with Lipschitz continuous gradient; see, for instance,
\cite{nesterov-2013}. However, as was showed in Section $3.2.1$, Theorem 3.2.1 of \cite{Book-Nest}, the worst expected error after $k$
iterations of the classical subgradient iteration is attainable equal to $\mathcal{O}\left((k+1)^{-1/2}\right)$ for general nonsmooth problems.  
\subsection{Exogenous stepsizes}
In this subsection we analyze the convergence of {\bf PSS
 Method}
using exogenous stepsizes, \emph{i.e.}, the positive exogenous sequence of
stepsizes $(\alpha_k)_{k\in \NN}$ satisfies that
$\dsty\alpha_k=\frac{\beta_k}{\eta_k}$ where
$\eta_k:=\max\{1,\|u^k\|\}$ for all $k$, and
\begin{equation}\label{pasos-exogenous}
\sum_{k=0}^{\infty} \beta_k^2<+\infty \qquad \mbox{and} \qquad
\sum_{k=0}^{\infty} \beta_k=+\infty.
\end{equation}
We begin with a useful consequence of Lemma \ref{lema-para-qF1}.
\begin{corollary}\label{lema-para-qF}
Let $x\in {\bf dom}(g)$. Then, for all $k\in \NN$,
$$\|x^{k+1}-x\|^2\le\|x^{k}-x\|^2
+2\frac{\beta_k}{\eta_k}\left[(f+g)(x)-(f+g)(x^k)\right]+\left(1+2\rho+\rho^2\right)\beta_k^2,$$
where $\rho\ge0$ is as defined in Assumption {\bf A2}.
\end{corollary}
\begin{proof}
The result follows by noting that $\eta_k\ge\|u^k\|$, $\eta_k\ge1$
for all $k\in \NN$ and letting $w^k\in \partial g(x^k)$ such that $\|w^k\|\le\rho$ for all $k\in\NN$ in view of Assumption {\bf A2}. Then,
$$
\frac{\|u^k+w^k\|^2}{\eta_k^2}\le\frac{\|u^k\|^2}{\eta_k^2}+2\frac{\|u^k\|\|w^k\|}{\eta_k^2}+\frac{\|w^k\|^2}{\eta_k^2}\le 1+2\rho+\rho^2.
$$ Now, Lemma~\ref{lema-para-qF1} implies the desired result.
\end{proof}
\noindent Now we define the auxiliary set
\begin{equation}\label{Slev} S_{\rm lev}(x^0):=\left\{x\in {\bf dom}(g) :
(f+g)(x)\le (f+g)(x^k),\; \forall k\in
\NN\right\}.\end{equation} When the solution set of problem
\eqref{prob} is nonempty, $S_{\rm lev}(x^0)\neq \emptyset$ because
$S_*\subseteq S_{\rm lev}(x^0)$. Next, we prove the two main results of this
subsection.

\begin{theorem}\label{ptos-de-acum2} Let $(x^k)_{k\in \NN}$ be the
sequence generated by {\bf PSS
Method} with exogenous stepsizes. If there exists $\bar{x}\in S_{\rm lev}(x^0)$, then:
\item [ {\bf(a)}] The sequence $(x^k)_{k\in \NN}$ is quasi-Fej\'er convergent to $$\mathcal{L}_{f+g}(\bar{x}):=\left\{x\in {\bf dom}(g) : (f+g)(x)\le
(f+g)(\bar{x})\right\}.$$
\item [ {\bf(b)}]$\lim_{k\rightarrow\infty}\,(f+g)(x^k)=(f+g)(\bar{x}).$
\item [ {\bf(c)}] The sequence $(x^k)_{k\in \NN}$ is weakly convergent to some $\tilde{x}\in
\mathcal{L}_{f+g}(\bar{x})$.
\end{theorem}
\begin{proof}
By assumption there exists $\bar{x}\in S_{\rm lev}(x^0)$, \emph{i.e.}, $(f+g)(\bar{x})\le (f+g)(x^k)$, for all $k\in \NN$.

\item [ {\bf(a)}] To show that $(x^k)_{k\in\NN}$ is quasi-Fej\'er convergent to
$\mathcal{L}_{f+g}(\bar{x})$ (which is nonempty because $\bar{x}\in
\mathcal{L}_{f+g}(\bar{x})$), we use Corollary \ref{lema-para-qF},
for any $x\in \mathcal{L}_{f+g}(\bar{x})\subseteq {\bf dom}(g)$,
establishing that $ \|x^{k+1}-x\|^2\le\|x^{k}-x\|^2
+(1+2\rho+\rho^2)\beta_k^2$, for all $k\in \NN$. Thus,
$(x^k)_{k\in \NN}$ is quasi-Fej\'er convergent to
$\mathcal{L}_{f+g}(\bar{x})$.
\item [ {\bf(b)}] The sequence $(x^k)_{k\in\NN}$ is
bounded from Fact \ref{fato-qF}(a), and hence it has accumulation
points in the sense of the weak topology. To prove that \begin{equation}\label{limite1_DI*}
\dsty\lim_{k\rightarrow\infty}\,(f+g)(x^k)=(f+g)(\bar{x}),
\end{equation} we use Corollary \ref{lema-para-qF}, with $x=\bar{x}\in
\mathcal{L}_{f+g}(\bar{x})\subseteq {\bf dom}(g)$, to get
\begin{align}\label{para-lim-to-02}\nonumber
\beta_k
\left[(f+g)(x^k)-(f+g)(\bar{x})\right]\le\,&\frac{1}{2}(\|x^{k}-\bar{x}\|^2-\|x^{k+1}-\bar{x}\|^2)
+\frac{1}{2}(1+2\rho+\rho^2)\beta_k^2.
\end{align} Summing, from $k=0$ to $m$, the above
inequality, we have \begin{align*}\sum_{k=0}^m\beta_k
\left[(f+g)(x^k)-(f+g)(\bar{x})\right]\le\,&\frac{1}{2}(
\|x^0-\bar{x}\|^2-\|x^{m+1}-\bar{x}\|^2)+\frac{1}{2}(1+2\rho+\rho^2)\sum_{k=0}^m\beta_k^2,\end{align*}
and taking limit, when $m$ goes to $\infty$,
\begin{equation}\label{converge-el-prod}\sum_{k=0}^{\infty}\beta_k\left[(f+g)(x^k)-(f+g)(\bar{x})\right]<+\infty.\end{equation}
Then, \eqref{converge-el-prod} together with \eqref{pasos-exogenous}
implies that there exists a subsequence
$\left((f+g)(x^{i_k})\right)_{k\in \NN}$ of
$\left((f+g)(x^k)\right)_{k\in \NN}$ such that
\begin{equation}\label{limite1_DI**}
\dsty\liminf_{k\rightarrow\infty}\,\left[(f+g)(x^{i_k})-(f+g)(\bar{x})\right]
= 0.
\end{equation} Indeed, if
\eqref{limite1_DI**} does not hold, then there exist $\sigma> 0$
and $k\ge\tilde{k}$, such that $(f+g)(x^k)-(f+g)(\bar{x})\ge\sigma$
and using \eqref{converge-el-prod}, we get
$$+\infty>\sum_{k=\tilde{k}}^{\infty}\beta_k
\left[(f+g)(x^k)-(f+g)(\bar{x})\right]\ge\sigma\sum_{k=\tilde{k}}^{\infty}\beta_k,
$$ in contradiction with \eqref{pasos-exogenous}. Next, define $\varphi_k:=(f+g)(x^k)-(f+g)(\bar{x})$, which is
positive for all $k$ because $\bar{x}\in S_{\rm lev}(x^0)$. Then, for any
$u^k\in
\partial f(x^k)$ and $w^k\in
\partial g(x^k)$, we get
\begin{align}\label{diferencia}\nonumber
\varphi_k-\varphi_{k+1}\,&=(f+g)(x^k)-(f+g)(x^{k+1})\le\la
u^k+w^k,x^k-x^{k+1}\ra\\\,&\le \|u^k+w^k\|\|x^k-x^{k+1}\|\le
(\zeta+\rho)\|x^k-x^{k+1}\|,
\end{align} where $\zeta>0$ such that $\|u^k\|\le \zeta$, for all $k\in\NN$ ($\zeta$ exists in virtue of the boundedness of
$(x^k)_{k\in \NN}$ and Assumption {\bf A1}) and
$\|w^k\|\le \rho$, for all $k\in\NN$ ($\rho$ exists because
$w^k\in\partial g(x^k)$ are arbitrary and the use of Assumption {\bf A2}).
Using Corollary \ref{lema-para-qF}, with $x=x^k$, we have
$\dsty\|x^k-x^{k+1}\|\le \sqrt{1+2\rho+\rho^2}\cdot\beta_k$, which
together with \eqref{diferencia} implies that
\begin{equation}\label{para-usar}\varphi_k-\varphi_{k+1}\le
\sqrt{1+2\rho+\rho^2}\cdot(\zeta+\rho)\beta_k:=\bar{\rho}\beta_k\end{equation}
for all $k\in\NN$. From \eqref{limite1_DI**}, there exists a
subsequence $\left(\varphi_{i_k}\right)_{k\in\NN}$ of
$\left(\varphi_k\right)_{k\in\NN}$ such that
$\lim_{k\rightarrow\infty}\varphi_{i_k}=0 $. If the claim given in
\eqref{limite1_DI*} does not hold, then there exists some $\delta>0$
and a subsequence $\left(\varphi_{\ell_k}\right)_{k\in\NN}$ of
$\left(\varphi_k\right)_{k\in\NN}$, such that
$\varphi_{\ell_k}\ge\delta$ for all $k\in\NN$. Thus, we can
construct a third subsequence $\left(\varphi_{j_k}\right)_{k\in\NN}$
of $\left(\varphi_k\right)_{k\in\NN}$, where the indices $j_k$ are
chosen in the following way:
$$\dsty j_0:=\dsty\min\{m\ge0 \mid \varphi_m\ge\delta\},$$
$$\dsty j_{2k+1}:=\dsty \min\{m\ge j_{2k}\mid \varphi_m\le\delta/2\},$$
$$ j_{2k+2}:=\dsty\min\{m\ge j_{2k+1}\mid \varphi_m\ge\delta\},$$
for each $k$. The existence of the subsequences
$\left(\varphi_{i_k}\right)_{k\in\NN}$,
$\left(\varphi_{\ell_k}\right)_{k\in\NN}$ of
$\left(\varphi_k\right)_{k\in \NN}$, guarantees that the subsequence
$\left(\varphi_{j_k}\right)_{k\in \NN}$ of
$\left(\varphi_k\right)_{k\in \NN}$ is well-defined for all $k\ge0$.
It follows from the definition of $j_k$ that
\begin{equation}\label{17}
\varphi_m\ge \delta \quad \mbox{for}\quad j_{2k}\le m\le j_{2k+1}-1
\end{equation}
\begin{equation*}\label{18}
\varphi_m\le \frac{\delta}{2} \quad \mbox{for}\quad j_{2k+1}\le m\le
j_{2k+2}-1
\end{equation*} for all $k$,
and hence
\begin{equation}\label{19}
\varphi_{j_{2k}}-\varphi_{j_{2k+1}}\ge \frac{\delta}{2},
\end{equation}
for all $k\in\NN$. In view of \eqref{converge-el-prod} and remind
that $\varphi_k=(f+g)(x^k)-(f+g)(\bar{x})\ge0$ for all $k\in \NN$,
\begin{eqnarray*}\label{desiguldad-imposible}\nonumber\dsty +\infty\,&>\,&\sum_{k=0}^{\infty} \beta_k\varphi_k\ge \sum_{k=0}^{\infty}
\sum_{m=j_{2k}}^{j_{2k+1}-1} \beta_m\varphi_m \ge \frac{\delta}{2}
\sum_{k=0}^{\infty} \sum_{m=j_{2k}}^{j_{2k+1}-1} \beta_m\\\,&=\,&
\frac{\delta}{2\bar{\rho}} \sum_{k=0}^{\infty}
\sum_{m=j_{2k}}^{j_{2k+1}-1} \bar{\rho} \beta_m \ge
\frac{\delta}{2\bar{\rho}} \sum_{k=0}^{\infty}
\sum_{m=j_{2k}}^{j_{2k+1}-1} (\varphi_m-\varphi_{m+1})
=\frac{\delta}{2\bar{\rho}} \sum_{k=0}^{\infty}
(\varphi_{j_{2k}}-\varphi_{j_{2k+1}})\\\,&\ge\,&\frac{\delta}{2\bar{\rho}}
\sum_{k=0}^{\infty}\frac{\delta}{2}=+\infty ,\end{eqnarray*} where
we have used \eqref{17} in the second inequality and
\eqref{para-usar} in the third inequality and \eqref{19} in the last
one. Thus,
$\dsty\lim_{k\rightarrow\infty}\,(f+g)(x^k)=(f+g)(\bar{x})$,
establishing {\bf(b)}.
\item [ {\bf(c)}] Let $\tilde{x}$ be a weak accumulation point of $(x^k)_{k\in\NN}$, and note that $\tilde{x}$ exists by Item {\bf(a)} and Fact \ref{fato-qF}(a).
From now on, we use $(x^{i_k})_{k\in \NN}$ to denote any
subsequence of $(x^k)_{k\in\NN}$ that converges weakly to
$\tilde{x}$. Since $f+g$ is weakly lower semicontinuous, using
\eqref{limite1_DI*}, we get
$$
(f+g)(\tilde{x})\le
\liminf_{k\rightarrow\infty}(f+g)(x^{i_k})=\lim_{k\rightarrow\infty}(f+g)(x^{k})=(f+g)(\bar{x}),
$$
implying that $(f+g)(\tilde{x})\le(f+g)(\bar{x})$ and thus
$\tilde{x}\in \mathcal{L}_{f+g}(\bar{x})$. As consequence, all weak
accumulation points of $(x^k)_{k\in \NN}$ belong to
$\mathcal{L}_{f+g}(\bar{x})$ and since $(x^k)_{k\in \NN}$ is
quasi-Fej\'er convergent to $\mathcal{L}_{f+g}(\bar{x})$, we get
that $(x^k)_{k\in \NN}$ converges weakly to $\tilde{x}\in
\mathcal{L}_{f+g}(\bar{x})$ from Fact \ref{fato-qF}(b).
\end{proof}

\begin{theorem}\label{ptos-de-acum2*} Let $(x^k)_{k\in \NN}$ be the
sequence generated by {\bf PSS
Method} with exogenous stepsizes. Then,
\item [ {\bf(a)}] $\liminf_{k\rightarrow\infty}(f+g)(x^k)=\inf_{x\in \HH}
(f+g)(x)=s_*$ (possibly $s_*=-\infty$).
\item [ {\bf(b)}]If $S_*\neq \emptyset$, then $\lim_{k\to\infty}(f+g)(x^k)=\min_{x\in\HH}(f+g)(x)$ and 
$(x^k)_{k\in\NN}$ converges weakly to some $\bar{x}\in
S_*$.
\item [ {\bf(c)}] If $S_*=\emptyset$, then $(x^k)_{k\in \NN}$ is
unbounded.
\end{theorem}
\begin{proof}
\item [ {\bf(a)}] Since $(x^k)_{k\in \NN}\subset {\bf dom}(g)$, we get
$s_*\le\liminf_{k\rightarrow\infty}(f+g)(x^k)$. Suppose that
$s_*<\liminf_{k\rightarrow\infty}(f+g)(x^k)$. Hence, there exists
$\hat{x}$ such that
\begin{equation}\label{lim-optimo}(f+g)(\hat{x})<\liminf_{k\rightarrow\infty}(f+g)(x^k).\end{equation}
It follows from \eqref{lim-optimo} that there exists $\bar{k}\in\NN$
such that $(f+g)(\hat{x})\le (f+g)(x^k)$ for all $k\ge \bar{k}$.
Since $\bar{k}$ is finite we can assume without loss of generality
that $(f+g)(\hat{x})\le (f+g)(x^k)$ for all $k\in \NN$. Using the
definition of $S_{\rm lev}(x^0)$, given in \eqref{Slev}, we have that
$\hat{x}\in S_{\rm lev}(x^0)$. By Theorem \ref{ptos-de-acum2}(b)
$\lim_{k\rightarrow\infty}(f+g)(x^k)=(f+g)(\hat{x})$, in
contradiction with \eqref{lim-optimo}. 

\item [ {\bf(b)}] Since $S_*\neq \emptyset$, take $x_*\in S_*$ and note that this implies $\mathcal{L}_{f+g}(x_*)=S_*$.
Since $(x^k)_{k\in \NN}\subset {\bf dom}(g)$, we get
$(f+g)(x_*)\le (f+g)(x^k)$ for all $k\in \NN$ implying that $x_*\in
S_{\rm lev}(x^0)$. By applying items (b) and (c) of Theorem \ref{ptos-de-acum2}, at $\bar{x}=x_*$, we get
that $\lim_{k\to\infty}(f+g)(x^k)=(f+g)(x_*)$ and $(x^k)_{k\in \NN}$ converges weakly to some
$\tilde{x}\in S_*$, respectively.

\item [ {\bf(c)}] Assume that $S_*$ is empty but $(x^k)_{k\in \NN}$ is bounded. Let $(x^{\ell_k})_{k\in \NN}$ be a subsequence of
$(x^k)_{k\in \NN}$ such that
$\lim_{k\rightarrow\infty}(f+g)(x^{\ell_k})=\liminf_{k\rightarrow\infty}(f+g)(x^k)$.
Since $(x^{\ell_k})_{k\in \NN}$ is bounded, without loss
of generality (\emph{i.e.}, refining $(x^{\ell_k})_{k\in \NN}$ if
necessary), we may assume that $(x^{\ell_k})_{k\in \NN}$
converges weakly to some $\bar{x}\in {\bf dom}(g)$. By the weak lower
semicontinuity of $f+g$ on ${\bf dom}(g)$,
\begin{equation}\label{ult-eq}(f+g)(\bar{x})\le\liminf_{k\rightarrow\infty}(f+g)(x^{\ell_k})
=\lim_{k\rightarrow\infty}(f+g)(x^{\ell_k})=\liminf_{k\rightarrow\infty}(f+g)(x^k)=s_*,\end{equation}
using Item {\bf(a)} in the last equality. By \eqref{ult-eq}, $\bar{x}\in
S_*$, in contradiction with the hypothesis and the result follows.
\end{proof}

For exogenous stepsizes, Theorem \ref{ptos-de-acum2*}(a) guarantees the
convergence of $\left((f+g)(x^k)\right)_{k\in \NN}$ to the optimal
value of problem \eqref{prob}, \emph{i.e.}, $
\liminf_{k\rightarrow\infty}(f+g)(x^k)=s_*, $ implying the
convergence of $\left((f+g)^k_{\rm best}\right)_{k\in \NN}$, defined
in \eqref{mejor-valor}, to $s_*$. It is important to mention that in
the proof of the above two crucial results, we have used a similar idea recently
presented in \cite{yun-sub} for a different instance.

\noindent In the following we present a direct consequence of Lemmas
\ref{f-best-alphak} and \ref{f-best-alphak2}, when the stepsizes satisfy \eqref{pasos-exogenous}.
\begin{corollary}
Let $(\bar{x}^k)_{k\in \NN}$ be the ergodic sequence
defined by \eqref{egodic-seq} and $\left(\beta_k\right)_{k\in \NN}$
as \eqref{pasos-exogenous}. If $S_*\neq \emptyset$, then, for all
$k\in\NN$,
\begin{equation*}
(f+g)_{\rm best}^k-\min_{x\in\HH}(f+g)(x)\le \zeta \frac{[{\rm
dist}(x^0,S_*)]^2+(1+2\rho+\rho^2)\sum_{i=0}^k \beta_i^2}{
2\sum_{i=0}^k \beta_i}
\end{equation*}
and
\begin{equation*}
(f+g)(\bar{x}^k)-\min_{x\in\HH}(f+g)(x)\le \zeta \frac{[{\rm
dist}(x^0,S_*)]^2+(1+2\rho+\rho^2)\sum_{i=0}^k \beta_i^2}{
2\sum_{i=0}^k \beta_i},
\end{equation*} where $\zeta>0$ and $\rho\ge0$ are as in Assumptions {\bf A1} and {\bf A2}, respectively.
\end{corollary}
The above corollary shows that if we assume existence of solutions,
the expected error of the iterates generated by {\bf PSS
 Method} with
the exogenous stepsizes \eqref{pasos-exogenous} after $k$ iterations
is $\mathcal{O}\left((\sum_{i=0}^k \beta_i)^{-1}\right)$. Since
$\left(\beta_k\right)_{k\in \NN}$ satisfies \eqref{pasos-exogenous}
the best performance of the iteration (in term of functional values)
is archived for example taking $\beta_k\cong1/k^r$ with $r$ bigger
than $1/2$, but near of this value, for all $k$.
\subsection{Polyak stepsizes}
In this subsection we analyze the convergence of {\bf PSS
 Method}
using Polyak stepsizes. Having chose any
$w^k\in
\partial g(x^k)$ and denoted $\rho_k:=\|w^k\|$ for all $k\in\NN$. Then define, for all
$k\in\NN$,
\begin{equation}\label{polyak-stepsize}
\alpha_k=\gamma_k\frac{(f+g)(x^k)-s_k}{\|u^k\|^2+2\rho_k\|u^k\|+\rho_k^2},
\end{equation} where $0<\gamma\le\gamma_k\le2-\gamma$. We assume that $s_k$ a monotone decreasing variable
target value approximating $s_*:=\inf\{(f+g)(x) :  x\in\HH\}$ is
available, and satisfies that $s_k\le (f+g)(x^k)$ for all $k\in\NN$.
When $s_*$ is known, the simplest variant of the stepsizes proposed
in \eqref{polyak-stepsize} is obtained selecting the stepsizes
\begin{equation}\label{iteracion-poljak}
\alpha_k=\gamma_k\frac{(f+g)(x^k)-s_*}{\|u^k\|^2+2\rho_k\|u^k\|+\rho_k^2},
\end{equation} for all $k\in\NN$. Unfortunately, to find an optimal solution, scheme \eqref{iteracion-poljak} requires prior knowledge of the optimal objective function
value $s_*$. As $s_*$ is usually unknown, we prefer to do our analysis over \eqref{polyak-stepsize}, and replace $s_*$ by the variable target value $s_k$. When $g$ is the indicator function of a closed and
convex set  further discussion about how to choose $s_k$ is
presented in the literature for problems where a good upper or lower
bound of the optimal objective function value is available; see, for
instance, \cite{held,KAC91,sherali-1997}.

\noindent Now we present a direct consequence of Lemma \ref{lema-para-qF1}. Denote $\mathcal{L}_{f+g}(s):= \left\{x\in {\bf dom}(g) :
(f+g)(x)\le s\right\}$.
\begin{corollary}\label{lema-para-qF2}
Suppose that $\lim_{k\rightarrow\infty}s_k=\tilde{s}\ge s_*$ and let
any $x\in \mathcal{L}_{f+g}(\tilde{s})$. Then,
$$\|x^{k+1}-x\|^2\le\|x^{k}-x\|^2
-\gamma(2-\gamma)\frac{\left[s_k-(f+g)(x^k)\right]^2}{\|u^k\|^2+2\rho_k\|u^k\|+\rho_k^2},$$
for all $k\in \NN$.
\end{corollary}
\begin{proof} Take
$x\in\mathcal{L}_{f+g}(\tilde{s})= \left\{x\in {\bf dom}(g) :
(f+g)(x)\le \tilde{s}\right\}$. Since $(s_k)_{k\in\NN}$ is a
monotone decreasing sequence convergent to $\tilde{s}$, which is
less than the function values of the iterates,
\begin{equation}\label{****}
(f+g)(x^k)\ge s_k\ge\tilde{s}\ge (f+g)(x), \;\; \forall x\in
\mathcal{L}_{f+g}(\tilde{s}),
\end{equation} for all $k\in\NN$. Then, applying Lemma \ref{lema-para-qF1} and using \eqref{****}, we
get, for all $k\in\NN$,
\begin{align}\label{polyak-ineq1}\nonumber
\|x^{k+1}-x\|^2\le\,&\|x^{k}-x\|^2-2\gamma_k\frac{\left[s_k-(f+g)(x^k)\right]\left[(f+g)(x)-(f+g)(x^k)\right]}{\|u^k\|^2+2\rho_k\|u^k\|+\rho_k^2}\\\nonumber
&+\gamma_k^2\frac{\left[s_k-(f+g)(x^k)\right]^2}{\|u^k\|^2+2\rho_k\|u^k\|+\rho_k^2}\\\nonumber
\le\,&
\|x^{k}-x\|^2-\gamma_k(2-\gamma_k)\frac{\left[s_k-(f+g)(x^k)\right]^2}{\|u^k\|^2+2\rho_k\|u^k\|+\rho_k^2}\\\le\,&
\|x^{k}-x\|^2-\gamma(2-\gamma)\frac{\left[s_k-(f+g)(x^k)\right]^2}{\|u^k\|^2+2\rho_k\|u^k\|+\rho_k^2},
\end{align} where we used that $x\in \mathcal{L}_{f+g}(\tilde{s})$, \eqref{polyak-stepsize} and \eqref{****} in the second inequality. The result follows from \eqref{polyak-ineq1}.
\end{proof}
\noindent Now, we prove the first main result of this subsection in
the following theorem.
\begin{theorem}\label{ptos-de-acum3} Let $(x^k)_{k\in \NN}$ be the
sequence generated by {\bf PSS
 Method} with $\alpha_k$ as in
\eqref{polyak-stepsize}. If  $\lim_{k\rightarrow\infty}
s_k=\tilde{s}\ge s_*$ and $\mathcal{L}_{f+g}(\tilde{s})\neq
\emptyset$, then
\item [ {\bf(a)}] $(x^k)_{k\in \NN}$ is Fej\'er convergent to $\mathcal{L}_{f+g}(\tilde{s})$.
\item [ {\bf(b)}] $\lim_{k\rightarrow\infty}\,(f+g)(x^k)=\tilde{s}.$
\item [ {\bf(c)}] $(x^k)_{k\in \NN}$ is weakly convergent to some $\tilde{x}\in
\mathcal{L}_{f+g}(\tilde{s})$.
\end{theorem}
\begin{proof}\item [ {\bf(a)}] It is direct consequence of Corollary
\ref{lema-para-qF2}.
\item [ {\bf(b)}] By Item {\bf(a)}, $(x^k)_{k\in \NN}$ is bounded. By using Corollary
\ref{lema-para-qF2}, at any $x\in
\mathcal{L}_{f+g}(\tilde{s})$, we get
\begin{align}\label{des-para-lim*}
\gamma(2-\gamma)\left[s_k-(f+g)(x^k)\right]^2\le\,&
\left(\|u^k\|^2+2\rho_k\|u^k\|+\rho_k^2\right)\left[\|x^{k}-x\|^2-\|x^{k+1}-x\|^2\right]\\
\le \,&\left(\zeta^2+2\rho
\zeta+\rho^2\right)\left[\|x^{k}-x\|^2-\|x^{k+1}-x\|^2\right]\nonumber\\:=\,&\hat{\rho}\left[\|x^{k}-x\|^2-\|x^{k+1}-x\|^2\right],\label{des-para-lim}\end{align}
where the last inequality following from Assumptions {\bf A1} and
{\bf A2} ($\|u^k\|\le \zeta$ and $\rho_k=\|w^k\|\le \rho $ for all $k\in\NN$).
Summing \eqref{des-para-lim}, over $k=0$ to $m$, we obtain
\begin{align*}\label{des-para-lim2}
\gamma(2-\gamma)\sum_{k=0}^m\left[s_k-(f+g)(x^k)\right]^2\le
\,&\hat{\rho}\left[\|x^{0}-x\|^2-\|x^{m+1}-x\|^2\right]\le
\hat{\rho}\|x^{0}-x\|^2 .\end{align*} Taking limit when $m$ goes to
$\infty$, we get the desired result.
\item[ {\bf(c)}] From Item {\bf(b)}, if $\tilde{s}=\lim_{k\rightarrow\infty}
s_k$ then $\lim_{k\rightarrow\infty}\,(f+g)(x^k)=\tilde{s}$. Let
$\tilde{x}$ be a weak accumulation point of
$(x^k)_{k\in\NN}$, which exists by the boundedness of
$(x^k)_{k\in\NN}$ direct consequence of Item {\bf(a)}. From
now on, we denote $(x^{\ell_k})_{k\in \NN}$ any
subsequence of $(x^k)_{k\in\NN}$, which converges weakly to
$\tilde{x}$. Since $f+g$ is weakly lower semicontinuous, we get $
(f+g)(\tilde{x})\le
\liminf_{k\rightarrow\infty}(f+g)(x^{\ell_k})=\lim_{k\rightarrow\infty}(f+g)(x^{k})=\tilde{s},
$ implying that $(f+g)(\tilde{x})\le\tilde{s}$ and thus
$\tilde{x}\in \mathcal{L}_{f+g}(\tilde{s})$. The result follows from
Fact \ref{fato-qF}(b) and Item {\bf(a)}.
\end{proof}

Before the analysis of the inconsistent case when
$\tilde{s}=\lim_{k\rightarrow\infty}s_k$ is strictly less than
$s_*=\inf\{(f+g)(x) :  x\in\HH\}$, we present a useful corollary which is a direct
consequence of Theorem \ref{ptos-de-acum3}, that shall be used for the analysis of
this case, $\tilde{s}<s_*$. In the next corollary, we show
the special case when the optimal value $s_*$ is known and finite and the stepsize $\alpha_k$ is
defined by \eqref{iteracion-poljak}, \emph{i.e.},  for all $k\in\NN$,
$$\dsty
\alpha_k=\gamma_k\frac{(f+g)(x^k)-s_*}{\|u^k\|^2+2\rho_k\|u^k\|+\rho_k^2},$$
where $0<\gamma\le\gamma_k\le2-\gamma$.
\begin{corollary}\label{ptos-de-acum4} Let $(x^k)_{k\in \NN}$ the
sequence generated by {\bf PSS
 Method} with $\alpha_k$ given by
\eqref{iteracion-poljak}, and $S_*\neq \emptyset$. Then,
\item [ {\bf(a)}] $(x^k)_{k\in \NN}$ is Fej\'er convergent to $S_*$.
\item [ {\bf(b)}] $\lim_{k\rightarrow\infty}\,(f+g)(x^k)=\min_{x\in\HH}(f+g)(x).$
\item [ {\bf(c)}] $(x^k)_{k\in \NN}$ is weakly convergent to some $\tilde{x}\in
S_*$.
\item [ {\bf(d)}] $\liminf_{k\rightarrow\infty}\,\sqrt{k+1}\cdot\left[(f+g)(x^k)-\min_{x\in\HH}(f+g)(x)\right]=0$.
\end{corollary}
\begin{proof}
Items {\bf(a)} to {\bf(c)} are direct consequence of Theorem
\ref{ptos-de-acum3}. The proof of Item {\bf(d)} is by contradiction.
Assume that
$\liminf_{k\rightarrow\infty}\,\sqrt{k+1}\cdot\left[(f+g)(x^k)-\min_{x\in\HH}(f+g)(x)\right]\ge 2\delta$,
for some $\delta>0$. Then, for $\bar{k}$ large enough, we have
$(f+g)(x^k)-\min_{x\in\HH}(f+g)(x) \ge \frac{\delta}{\sqrt{k+1}}$ for all $k\ge
\bar{k}$. Thus,
\begin{equation}\label{ultima-eq}
\sum_{k=\bar{k}}^\infty \left[(f+g)(x^k)-\min_{x\in\HH}(f+g)(x)\right]^2 \ge\delta^2
\sum_{k=\bar{k}}^\infty\frac{1}{k+1}=+\infty. \end{equation} On the
other hand, by substituting the expression for the stepsize
$\alpha_k$ given by \eqref{iteracion-poljak}, in
\eqref{des-para-lim} ($s_k=\min_{x\in\HH}(f+g)(x)$ for all $k\in\NN$), we get, for all
$k\ge \bar{k}$,
$$
\sum_{k=\bar{k}}^\infty \left[(f+g)(x^k)-\min_{x\in\HH}(f+g)(x)\right]^2<+\infty,
$$ which contradicts \eqref{ultima-eq} thus, establishing the result.
\end{proof}
\noindent Next we present a result on the complexity of the iterates.
\begin{lemma}\label{comp-polyak}Let $(x^k)_{k\in \NN}$ be the
sequence generated by {\bf PSS
 Method} with $\alpha_k$, given by
\eqref{polyak-stepsize}. If  $\lim_{k\rightarrow\infty}
s_k=\tilde{s}\ge s_*$ and $\mathcal{L}_{f+g}(\tilde{s})\neq
\emptyset$, then, for all $k\in \NN$,
\begin{equation*}\label{des-para-lim2*lema}
(f+g)^k_{\rm best}-\tilde{s}\le
\sqrt{\frac{D_k}{\gamma(2-\gamma)}}\cdot\frac{{\rm
dist}(x^0,\mathcal{L}_{f+g}(\tilde{s}))
}{\sqrt{k+1}},\end{equation*}  where $D_k:=\max\left\{\|u^i\|^2+2\rho_i\|u^i\|+\rho_i^2 : 1\le i\le
k\right\}$ with
$\rho_i:=\|w^i\|$ and $w^i\in \partial g(x^i)$ $(i=0,\ldots,k)$ are
arbitrary. Moreover,
$$
\lim_{k\rightarrow\infty}\, (f+g)^k_{\rm best}=\tilde{s}.
$$
\end{lemma}
\begin{proof}
Repeating the proof of Theorem \ref{ptos-de-acum3}, with $\tilde{x}:={\bf P}_{\mathcal{L}_{f+g}(\tilde{s})}(x^0)\in
\mathcal{L}_{f+g}(\tilde{s})$, until \eqref{des-para-lim*}, we obtain
\begin{align*}\label{des-para-lim2*}
(k+1)\left[(f+g)^k_{\rm best}-\tilde{s}\right]^2\le
\sum_{i=0}^k\left[(f+g)(x^i)-s_k\right]^2\le
\frac{D_k}{\gamma(2-\gamma)}\left[{\rm
dist}(x^0,\mathcal{L}_{f+g}(\tilde{s}))\right]^2 ,\end{align*} where
$D_k:=\max\left\{\|u^i\|^2+2\rho_i\|u^i\|+\rho_i^2 : 1\le i\le
k\right\}$ with
$\rho_i=\|w^i\|$ and $w^i\in \partial g(x^i)$ $(i=0,\ldots,k)$ are
arbitrary. After simple algebra the result follows.
\end{proof}
Our analysis proved that the expected error of the iterates
generated by {\bf PSS
 Method} with the Polyak stepsizes
\eqref{polyak-stepsize} after $k$ iterations is
$\mathcal{O}\left((k+1)^{-1/2}\right)$ if we assume $s_k\ge s_*$ for
all $k\in\NN$. 

\noindent Now we are ready to prove the last main result of this
subsection.
\begin{theorem}\label{ptos-de-acum33} Let $(x^k)_{k\in \NN}$ be the
sequence generated by {\bf PSS
 Method} with $\alpha_k$, given by
\eqref{polyak-stepsize}. If $S_*\neq \emptyset$ and
$\lim_{k\rightarrow\infty} s_k=\tilde{s}< \min_{x\in\HH}(f+g)(x)$, then
$$
\lim_{k\rightarrow\infty}(f+g)^k_{\rm
best}=\lim_{k\rightarrow\infty}\min_{0\le i\le k} (f+g)(x^i)\le
\min_{x\in\HH}(f+g)(x)+\frac{2-\gamma}{\gamma}\left[\min_{x\in\HH}(f+g)(x)-\tilde{s}\right].
$$
\end{theorem}
\begin{proof}
Suppose that $(f+g)(x^k)>\min_{x\in\HH}(f+g)(x)$, otherwise the result holds trivially.
It is clear that, for all $k\in \NN$,
\begin{align*}
\alpha_k&=\gamma_k\frac{(f+g)(x^k)-s_k}{(f+g)(x^k)-\min_{x\in\HH}(f+g)(x)}\frac{(f+g)(x^k)-\min_{x\in\HH}(f+g)(x)}{\|u^k\|^2+2\rho_k\|u^k\|+\rho_k^2}\\
&:=\tilde{\gamma}_k\frac{(f+g)(x^k)-\min_{x\in\HH}(f+g)(x)}{\|u^k\|^2+2\rho_k\|u^k\|+\rho_k^2},
\end{align*} where
$$\gamma\le \tilde{\gamma}_k=\gamma_k\frac{(f+g)(x^k)-s_k}{(f+g)(x^k)-\min_{x\in\HH}(f+g)(x)},$$
which implies that $\tilde{\gamma}_k$ is greater than $2-\gamma$ for some $\bar{k}\in \NN$.
Otherwise, if
\begin{equation}\label{eq-bon}\tilde{\gamma}_k\le 2-\gamma \end{equation} for all $k\in\NN$, we can apply Corollary
\ref{ptos-de-acum4}(b) to get
$\lim_{k\rightarrow\infty}\,(f+g)(x^k)=\min_{x\in\HH}(f+g)(x)$, which implies that
$\tilde{\gamma}_k$ goes to $+\infty$ (note that for all sufficiently
large $k$, $s_k<\min_{x\in\HH}(f+g)(x)\le (f+g)(x^k)$, because $\tilde{s}<\min_{x\in\HH}(f+g)(x)$), which
is a contradiction with \eqref{eq-bon}. Thus, there exist $\bar{k}$
and $\delta>0$ arbitrary such that
$$
\gamma_{\bar{k}}\frac{(f+g)(x^{\bar{k}})-s_{\bar{k}}}{(f+g)(x^{\bar{k}})-\min_{x\in\HH}(f+g)(x)}=\tilde{\gamma}_{\bar{k}}>2-\delta.
$$ After simple algebra and using that $s_{\bar{k}}\ge \tilde{s}$, we get that
$$
(f+g)(x^{\bar{k}})<
\min_{x\in\HH}(f+g)(x)+\frac{\gamma_{\bar{k}}}{2-\delta-\gamma_{\bar{k}}}[\min_{x\in\HH}(f+g)(x)-\tilde{s}]\le
\min_{x\in\HH}(f+g)(x)+\frac{2-\gamma}{\gamma-\delta}[\min_{x\in\HH}(f+g)(x)-\tilde{s}],
$$ since $\delta>0$ was arbitrary and the result follows.
\end{proof}
Finally in the following corollary we summarize the behaviour of the
limit of the sequence of $\left((f+g)^k_{\rm best}\right)_{k\in
\NN}$ depending on the limit of
$\tilde{s}=\lim_{k\rightarrow\infty}s_k$, which is direct
consequence of Theorems~\ref{ptos-de-acum3}(b) and \ref{ptos-de-acum33} and
Lemma~\ref{comp-polyak}.
\begin{corollary}Let $(x^k)_{k\in \NN}$ be the
sequence generated by {\bf PSS
 Method} with $\alpha_k$, given by
\eqref{polyak-stepsize}. If $S_*\neq \emptyset$ and
$\lim_{k\rightarrow\infty} s_k=\tilde{s}$, then
$$\lim_{k\rightarrow\infty}\, (f+g)^k_{\rm best}\left\{\begin{array}{lll}=\dsty\lim_{k\rightarrow\infty}(f+g)(x^k)=\dsty \tilde{s},&{\mbox if}&\tilde{s}\ge\dsty \min_{x\in\HH}(f+g)(x)\\[2.2ex]
\le\dsty \min_{x\in\HH}(f+g)(x)+\frac{2-\gamma}{\gamma}\left[\min_{x\in\HH}(f+g)(x)-\tilde{s}\right],&{\mbox if}&\tilde{s}<\dsty\min_{x\in\HH}(f+g)(x).\end{array}\right.$$
\end{corollary}

\section{Final Remarks}\label{s:6}

In this work we dealt with the weak convergence and the complexity of the new approach
called the Proximal Subgradient Splitting (PSS) Method for
minimizing the sum of two nonsmooth and convex functions. In the
iteration of this method, none of the functions need be
differentiable or finite on $\HH$ and, therefore, a broad class
of problems can be solved. {\bf PSS Method} is very useful when the
proximal operator of $f$ is complex to evaluate and its
(sub)gradient is simple to compute.

As future research, we will investigate variations of our scheme for
solving structured convex optimization problems with the aim of
finding new methods, like the coordinate gradient method, which have
been proposed, for instance, in
\cite{nest-SIAM} only for the differentiable
case. We also look at the incremental subgradient method
\cite{nedic-SIAM-2001,aprox-kiwiel} for problem \eqref{prob}, when $f$ is the sum
of a large number of nonsmooth convex functions. The idea is to
perform subgradient iterations incrementally, by sequentially
taking steps along the subgradients of the component functions,
followed by proximal steps. 
%This incremental approach has been very
%successful in solving large least squares problems, and it has
%resulted in a much better practical rate of convergence with
%application in the training of neural networks; see, for instance,
%\cite{GA-NN}. 
On the other hand, it is important to mention that the main drawback of
subgradient iterations is their slow rate of convergence. However,
subgradient methods are distinguished by their applicability,
simplicity and efficient use of memory, which is very important for
large scale problems; especially if the required accuracy for the
solution is not too high; see, for instance, \cite{nesterov-2014}
and the references therein. We also will intend to study fast and
variable metric versions of the proximal subgradient splitting
method proposed here to achieve better performance, as in the
differentiable case; see \cite{comb-2014}.

Finally, we hope that this study serves as a basis for future
research on other more efficient variants on the proximal
subgradient iteration, like cutting-plane method,
$\epsilon$-subgradients and proximal bundle method and its variations; see \cite{claudia, aprox-kiwiel, eficient-kiwiel-I}.
Moreover, in future work we discuss useful modifications on the proximal subgradient iteration adding conditional, ergodic
and deflected techniques combining the ideas presented in
\cite{cond,francolli}.

\subsubsection*{ACKNOWLEDGMENTS} This work was completed while the author was visiting the University of British Columbia.
The author is very grateful for the warm hospitality of the Irving
K. Barber School of Arts and Sciences at the University of British
Columbia Okanagan and particularly to Professors Heinz H. Bauschke and Shawn Wang for the generous hospitality. The author would like to thank to anonymous
referees
whose suggestions helped us to improve the presentation of
this paper.
\bibliographystyle{plain}

\end{document}